\documentclass[oneside,french,english]{amsart}
\usepackage[T1]{fontenc}
\usepackage[latin9]{inputenc}
\usepackage{amsthm}
\usepackage{amssymb}
\usepackage{graphicx}
\usepackage{esint}

\makeatletter
\numberwithin{equation}{section}
\numberwithin{figure}{section}
\theoremstyle{plain}
\newtheorem{thm}{\protect\theoremname}
  \theoremstyle{plain}
  \newtheorem{lem}[thm]{\protect\lemmaname}
  \theoremstyle{plain}
  \newtheorem{prop}[thm]{\protect\propositionname}
  \theoremstyle{remark}
  \newtheorem{rem}[thm]{\protect\remarkname}
  \theoremstyle{plain}
  \newtheorem{cor}[thm]{\protect\corollaryname}
  \theoremstyle{definition}
  \newtheorem{defn}[thm]{\protect\definitionname}

\usepackage{color,graphics}

\AtBeginDocument{
  
}

\makeatother

\usepackage{babel}
\addto\extrasfrench{%
   \providecommand{\fg}{\ifdim\lastskip>\z@\unskip\fi~\frqq}%
}

  \addto\captionsenglish{\renewcommand{\corollaryname}{Corollary}}
  \addto\captionsenglish{\renewcommand{\definitionname}{Definition}}
  \addto\captionsenglish{\renewcommand{\lemmaname}{Lemma}}
  \addto\captionsenglish{\renewcommand{\propositionname}{Proposition}}
  \addto\captionsenglish{\renewcommand{\remarkname}{Remark}}
  \addto\captionsenglish{\renewcommand{\theoremname}{Theorem}}
  \addto\captionsfrench{\renewcommand{\corollaryname}{Corollaire}}
  \addto\captionsfrench{\renewcommand{\definitionname}{Définition}}
  \addto\captionsfrench{\renewcommand{\lemmaname}{Lemme}}
  \addto\captionsfrench{\renewcommand{\propositionname}{Proposition}}
  \addto\captionsfrench{\renewcommand{\remarkname}{Remarque}}
  \addto\captionsfrench{\renewcommand{\theoremname}{Théorème}}
  \providecommand{\corollaryname}{Corollary}
  \providecommand{\definitionname}{Definition}
  \providecommand{\lemmaname}{Lemma}
  \providecommand{\propositionname}{Proposition}
  \providecommand{\remarkname}{Remark}
\providecommand{\theoremname}{Theorem}

\begin{document}
\global\long\def\EQ#1#2{\raisebox{+.65ex}{\ensuremath{#1}}/\raisebox{-.65ex}{\ensuremath{#2}}}

\title{Classification of absolutely dicritical foliations of cusp type.}

\author{Y. Genzmer\\
02/03/2012}
\begin{abstract}
We give a classification of absolutely dicritical foliation of cusp
type, that is, the germ of singularities of complex foliation in the
complex plane topologically equivalent to the singularity given by
the level of the meromorphic function $\frac{y^{2}+x^{3}}{xy}$.
\end{abstract}
\maketitle
An important problem of the theory of singularities of holomorphic
foliations in the complex plane is the construction of a geometric
interpretation of the so-called \emph{moduli of Mattei} of these foliations
\cite{univ}. These moduli appear when one considers a very special
kind of deformations called \emph{the unfoldings}. Basically, the
moduli of Mattei are precisely the moduli of germs of unfoldings of
a given singular foliation. One of the major difficulty one meets
looking at the mentionned geometric description is the lack of basic
examples in the litterature. Actually, except when the foliation is
given by the level of an holomorphic function, there exist none exemple.
The purpose of the following article is not to solve the problem of
Mattei even for the class of singularities we consider here but to
describe this one as accuratly as possible in order to prepare the
attack of the problem of moduli of Mattei. 

The absolutely dicritical foliations of cusp type are good candidates
to begin this study for the following reasons:
\begin{enumerate}
\item their\emph{ transversal structure}, which usually is a very rich dynamic
invariant \cite{Loray}, is very poor and can be completely understood.
\item their number of Mattei moduli is $1.$
\item the topology of their leaves is more or less trivial.
\end{enumerate}
Some results in the article might be quite easily extended to a larger
class of absolutely dicritical foliations up to some technical and
confusing additions. The risk would have been to miss the very first
objective of this paper, that is,\emph{ to give an example.} 

\bigskip{}

A germ of singularity of foliation $\mathcal{F}$ in $\left(\mathbb{C}^{2},0\right)$
is said to be \emph{absolutely dicritical} if there exists a sequence
of blowing-up $E$ such that $E^{*}\mathcal{F}$ is regular and transverse
to each irreducible component of the exceptionnal divisor $E^{-1}\left(0\right)$.
It is of \emph{cusp type} if two successive blowing-up are sufficient.
In that case the exceptionnal divisor $E^{-1}\left(0\right)$ is the
union of two irreducible components $\mathbb{P}_{1}\left(\mathbb{C}\right)$
of respective self-intersection $-2$ and $-1.$ We denote them respectively
$D_{2}$ and $D_{1}$. 

\noindent \begin{center}
\includegraphics[scale=0.6]{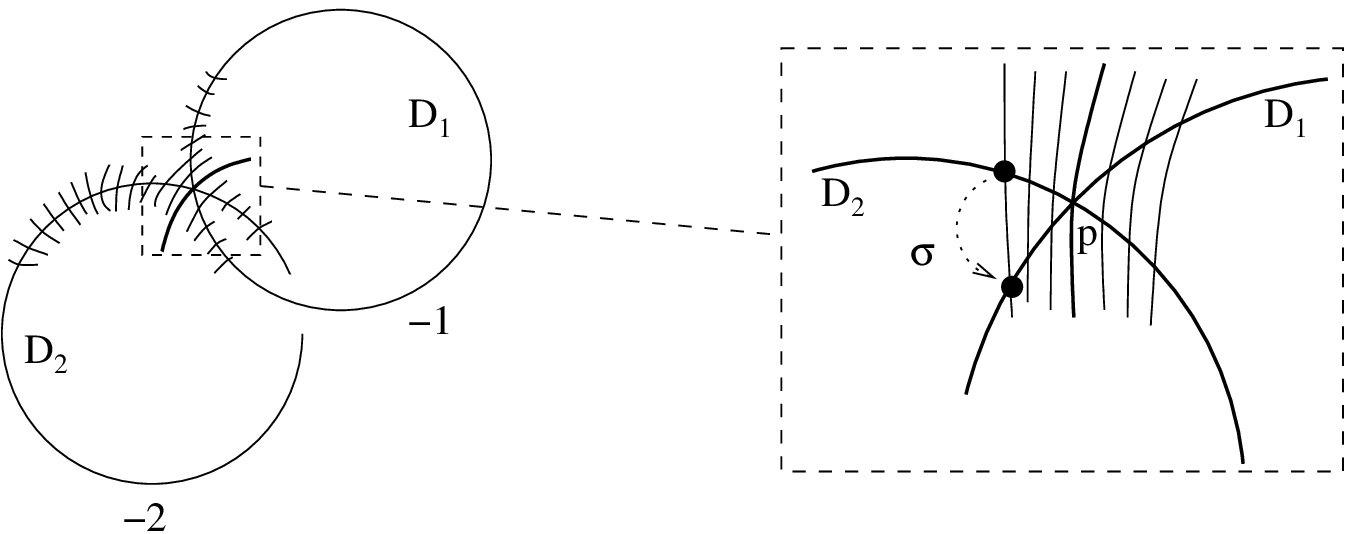}
\par\end{center}

The expression \emph{cusp type} insists on the fact that the special
leaf that passes trough the singular point of the divisor is analytically
equivalent to the cuspidal singularity $y^{2}+x^{3}=0.$ The simplest
example of an absolutely dicritical foliation is given by the levels
of the rationnal function near $\left(0,0\right)$ 
\[
f=\frac{y^{2}+x^{3}}{xy}.
\]

We associate to $\mathcal{F}$ a germ $\sigma\in\mbox{Diff}\left(\left(D_{2},p\right),\left(D_{1},p\right)\right)$
as in the picture above. It is defined by the property that $x\in D_{2}$
and $\sigma\left(x\right)\in D_{1}$ belongs to the same local leaf.
This germ is called \emph{the transversal structure} of $\mathcal{F}$.
This is the very first invariant of such a foliation. For the rationnal
function above, the transversal structure $\sigma$ reduces to the
identity map in the standard coordinates associated to $E$. 

\bigskip{}

The main result of this article is the following one: for any foliation
$\mathcal{F}$ that is absolutely dicritical of cusp type we consider
it topological class $\mbox{Top}\left(\mathcal{F}\right),$ that is
the set of all foliations topologically equivalent to $\mathcal{F}.$
The moduli space $\EQ{\mbox{Top}\left(\mathcal{F}\right)}{\sim}$
of $\mathcal{F}$ is defined as the quotient of $\mbox{Top}\left(\mathcal{F}\right)$
by the analytical equivalence relation. Now we have,
\begin{thm}
The class $\mbox{Top}\left(\mathcal{F}\right)$ is equal to the set
of all absolutely dicritical foliations and its moduli space $\EQ{\mbox{Top}\left(\mathcal{F}\right)}{\sim}$
can be identified with the functionnal space $\mathbb{C}\left\{ z\right\} $
up to the action of $\mathbb{C}^{*}$ defined by 
\[
\epsilon\cdot\left(z\mapsto f\left(z\right)\right)=\epsilon^{2}f\left(\epsilon z\right).
\]

\end{thm}
In this theorem, the germ of convergent series $f$ is the image of
the transversal structure $\sigma$ by the Schwarzian derivative $S\left(\sigma\right)=\frac{3}{2}\left(\frac{\sigma^{'''}}{\sigma^{'}}\right)-\left(\frac{\sigma^{''}}{\sigma^{'}}\right)^{2}.$
A quick lecture of the theorem would suggest that the transversal
structure $\sigma$ is the sole invariant of the foliation, which
is not exactly true as it is highlighted in theorem (\ref{prop:Two-normal-forms}). 

\bigskip{}

We have to mention that it does exists \emph{a lot of absolutely dicritical
foliation}s. Following a result due to F. Cano and N. Corral \cite{Caco},
the process $E$ does not contain any obstruction to the existence
of absolutely dicritical foliations. In other words, for any sequence
of blowing-up $E,$ there exists an absolutely dicritical foliation
whose associated process of blowing-ups is exactly $E$.

\section{Topological classification. }

The topological classification is\emph{ trivial} as stated in a proposition
to come in the sense that two absolutely dicritical foliations of
cusp type are topologically equivalent. To prove this fact, we describe
below the \emph{model foliations} from which the absolutely dicritical
foliations are build.

\subsection{Model foliations.}

Let us consider the following model foliations 
\begin{itemize}
\item $\mathcal{F}_{2}$ is given by the gluing of two copies of $\mathbb{C}^{2}$
\[
\mathbb{C}^{2}=\left(x_{1},y_{1}\right)\qquad\mathbb{C}^{2}=\left(x_{2},y_{2}\right)
\]
glued by $x_{2}=\frac{1}{y_{1}}$ and $y_{2}=y_{1}^{2}x_{1}$ whose
the neighborhood of $x_{1}=y_{2}=0$ is transversaly foliated by $y_{1}=cst$
and $x_{2}=cst$. Topologically, this is a foliated neighborhood of
a Riemann surface of genus $0$ whose self-intersection is $-2$. 
\item $\mathcal{F}_{1}$ is given by the gluing of two copies of $\mathbb{C}^{2}$
\[
\mathbb{C}^{2}=\left(x_{3},y_{3}\right)\qquad\mathbb{C}^{2}=\left(x_{4},y_{4}\right)
\]
 glued by $x_{4}=\frac{1}{y_{3}}$ and $y_{4}=y_{3}x_{3}$ whose the
neighborhood of $x_{3}=y_{4}=0$ is transversaly foliated by $y_{3}=cst$
and $x_{4}=cst$. Topologically, this is a foliated neighborhood of
a Riemann surface of genus $0$ whose self-intersection is 1. 
\end{itemize}
Following \cite{CamaMovaSad}, any neighborhood of a Riemann surface
$A$ of genus $0$ embedded in a manifold of dimension two with $A\cdot A=-2$
(resp. $-1$) and foliated by a transverse codimension $1$ foliation
is equivalent ot $\mathcal{F}_{2}$ (resp. $\mathcal{F}_{1}$). From
this, it is easy to show that any $\left(\mathcal{C}^{0},\mathcal{C}^{\infty},\mathcal{C}^{\omega}\right)-$isomorphism
between two Riemann surfaces $A_{1}$ and $A_{2}$ as before can be
extended in a neigborhood of $A_{1}$ and $A_{2}$ as a $\left(\mathcal{C}^{0},\mathcal{C}^{\infty},\mathcal{C}^{\omega}\right)-$
conjugacy of the foliations.

\subsection{Topological classification.}

Let us first recall the following lemma:
\begin{lem}
Let $\sigma$ be a germ in $\mbox{Diff}\left(\mathbb{P}^{1},a\right)$,
i.e., a germ of automorphism of a neighborhood of $a$ in $\mathbb{P}^{1}.$
Then there exists $h$ a global homeomorphism of $\mathbb{P}^{1}$
such that $h$ and $\sigma$ coincide in a neighborhood of $a$. \end{lem}
\begin{proof}
Let $S_{1}$ be a small circle around $a$ in a domain where $\sigma$
is defined. Its image $\sigma\left(S_{1}\right)$ is a topological
circle. Consider $S_{2}$ a second circle such that the disc bounded
by $S_{2}$ contains $S_{1}$ and $\sigma\left(S_{1}\right).$ The
two coronas bounded respectively by $S_{1}$ and $S_{2}$ and $\sigma\left(S_{1}\right)$
and $S_{2}$ are homeomorphic. Actually, there exists an homeomorphism
$\tilde{h}$ of the two coronas such that 
\begin{eqnarray*}
\left.\tilde{h}\right|_{S_{2}} & = & \mbox{Id}\\
\left.\tilde{h}\right|_{S_{1}} & = & \sigma.
\end{eqnarray*}
Therefore, we can define the homeomorphism $h$ the following way:
in the disc bounded by $S_{1}$, we set $h=\sigma;$ in the corona
bounded by $S_{1}$ and $S_{2}$, $h=\tilde{h}$; everywhere else
we set $h=\mbox{Id}$. Clearly, $h$ satifies the properties in the
lemma.\end{proof}
\begin{prop}
Two absolutely dicritical foliations of cusp type are topologically
equivalent. The class $\mbox{Top}\left(\mathcal{F}\right)$ is equal
to the set of all absolutely dicritical foliations. \end{prop}
\begin{proof}
Let us consider $\mathcal{F}_{0}$ and $\mathcal{F}_{1}$ two absolutely
dicritical foliations of cusp type. Applying if necessary a linear
change of coordinates to $\mathcal{F}_{0}$ for instance, we can suppose
that both foliations are reduced by exactly the same sequence of two
blowing-ups $E.$ Let us write $E^{-1}\left(0\right)=D_{2}\cup D_{1}$
and $D_{2}\cap D_{1}=\left\{ p\right\} .$ Let us consider $\sigma_{0}$
and $\sigma_{1}$ in $\mbox{Diff}\left(\left(D_{2},p\right),\left(D_{1},p\right)\right)$
the transversal structures of $\mathcal{F}_{0}$ and $\mathcal{F}_{1}.$
According to the previous lemma, there exist $h$ an homeomorphism
of $D_{2}$ such that $h=\sigma_{0}^{-1}\circ\sigma_{1}$ in a neighborhood
of $p$ in $D_{2}.$ Since, along $D_{2}$ or $D_{1}$ the foliations
are transverse, there exist two homeomorphisms $H_{0}$ and $H_{1}$
defined respectively in a neighborhood of $D_{2}$ and $D_{1}$ such
that 
\[
H_{0}^{*}\left(E^{*}\mathcal{F}_{0}\right)=E^{*}\mathcal{F}_{1}\quad H_{1}^{*}\left(E^{*}\mathcal{F}_{0}\right)=E^{*}\mathcal{F}_{1}
\]
 and $\left.H_{0}\right|_{D_{2}}=\mbox{Id}$ and $\left.H_{1}\right|_{D_{1}}=h$.
Since $h=\sigma_{0}^{-1}\circ\sigma_{1}$, the automorphism $H_{1}\circ H_{0}^{-1}$
of $E^{*}\mathcal{F}_{0}$ let invariant each leaf of $E^{*}\mathcal{F}_{0}$
. Now, adapting the argument of the previous lemma yields the existence
of $H$ a global homeomorphism of $E^{*}\mathcal{F}_{0}$ defined
in a neighborhood of $D_{1}$ letting invariant each leaf such that
$H$ and $H_{1}\circ H_{0}^{-1}$ coincide in a neighborhood of $p$.
Thus $\left(H^{-1}\circ H_{1}\right)\circ H_{0}^{-1}$ is equal to
$\mbox{Id}$ in a neighborhood of $p.$ Therefore the collection $H^{-1}\circ H_{1}$
and $H_{0}$ glue in a global homeomorphism between $E^{*}\mathcal{F}_{0}$
and $E^{*}\mathcal{F}_{1}$. This homeomorphism can be blown down
in a neighborhood of $\mathbb{C}^{2}$ and is a $\mathcal{C}^{0}$-
conjugacy of the foliations $\mathcal{F}_{0}$ and $\mathcal{F}_{1}.$ 

Now, if $\mathcal{F}_{0}$ is topologically equivalent to an absolutely
dicritical foliation of cusp type, a theorem of C. Camacho and A.
Lins Neto and P. Sad \cite{camacho} ensures that the process of reduction
of $\mathcal{F}_{0}$ is the one of an absolutely dicritical foliation.
Since, they also shared the same dicritical components, $\mathcal{F}_{0}$
is absolutely dicritical of cusp type. 
\end{proof}

\section{Moduli space. }

Consider a germ of biholomorphism $\phi$ written in the coordinates
of the model foliations 
\[
\left(x_{3},y_{3}\right)=\phi\left(x_{1},y_{1}\right),\qquad\phi\left(0,0\right)=\left(0,0\right).
\]
Suppose that it send the foliation defined by $y_{1}=cst$ to the
one defined by $y_{3}=cst$ and that the curve $x_{1}=0$ is send
to a curve transverse to $x_{3}=0$. With such a biholomorphism we
can consider the foliation obtained by gluing the two models foliations
$\mathcal{F}_{2}$ and $\mathcal{F}_{1}$ with the application $\phi$

\noindent \begin{center}
\includegraphics[scale=0.4]{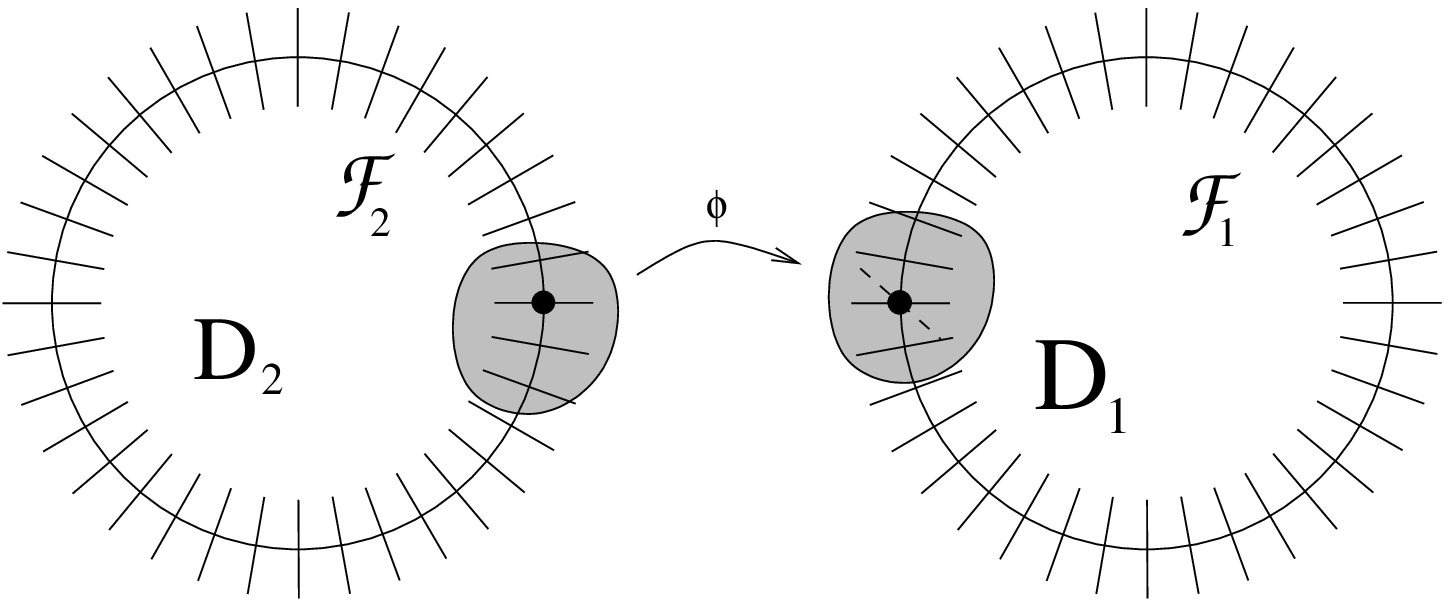}
\par\end{center}

Following a classical result due to Grauert, this gluing is analytically
equivalent to the neighborhood of the exceptionnal divisor obtained
by a standard process of two successive blowing-ups \cite{grauerthans}.
The obtained foliation can be blown down in an absolutely dicritical
foliation of cusp type at the origin of $\mathbb{C}^{2}.$ 
\begin{rem}
\emph{\label{Key-remark.-Two}Key remark.} Two foliations obtained
by such an above gluing with the respective biholomorphisms $\phi$
and $\psi$ are analytically equivalent if and only if there exists
an automorphism $\Phi_{2}$ of the foliation $\mathcal{F}_{2}$ and
$\Phi_{1}$ of the foliation $\mathcal{F}_{1}$ such that 
\[
\phi=\Phi_{1}\circ\psi\circ\Phi_{2}.
\]

\end{rem}
Let us fix $\sigma\in\mbox{Diff}\left(\mathbb{C},0\right)$. We consider
the following biholomorphisms 
\[
g_{\sigma}\left(x_{1},y_{1}\right)=\left(x_{1}+\sigma\left(y_{1}\right),\sigma\left(y_{1}\right)\right)\mbox{ and }\Phi_{\alpha}\left(x_{1},y_{1}\right)=\left(x_{1}\left(1+\alpha y_{1}\right),y_{1}\right)
\]

The composition $g_{\sigma}\circ\phi_{\alpha}$ send the foliation
$y_{1}=cst$ on it-self. Thus we can denote by $\mathcal{F}_{\sigma,\alpha}$
the foliation obtained by the above gluing

\[
\mathcal{F}_{\sigma,\alpha}:=\EQ{\mathcal{F}_{1}\coprod\mathcal{F}_{2}}{p\sim g_{\sigma}\circ\Phi_{\alpha}\left(p\right)}.
\]

Now, moving the parameter $\alpha$, we obtain an analytical family
of absolutely dicritical foliations. Actually, the following property
holds.
\begin{thm}
The germ of deformation $\left(\mathcal{F}_{\sigma,\alpha}\right)_{\alpha\in\left(\alpha_{0},\mathbb{C}\right)}$
for $\alpha$ in a neighborhood of $\alpha_{0}$ in $\mathbb{C}$
is a germ of equisingular semi-universal unfolding of $\mathcal{F}_{\sigma,\alpha_{0}}$
in the sense of Mattei \cite{univ}. In particular, for any germ of
equisingular unfolding $\left(\mathcal{F}_{t}\right)_{t\in\left(\mathbb{C}^{p},0\right)}$
with $p$ parameters such that $\left.\mathcal{F}_{t}\right|_{t=0}\sim\mathcal{F}_{\sigma,\alpha_{0}}$
there exists a map $\alpha:\left(\mathbb{C},\alpha_{0}\right)\to\left(\mathbb{C}^{p},0\right)$
such that for all $t$ $\mathcal{F}_{t}\sim\mathcal{F}_{\sigma,\alpha\left(t\right)}$. 
\end{thm}
Before proving the above result, let us recall that an \emph{unfolding}
of a given foliation $\mathcal{F}$ is a germ $\mathbb{F}$ of codimension
$1$ foliation in $\left(\mathbb{C}^{2+p},0\right)$ transversal to
the fiber of the projection $\left(\mathbb{C}^{2+p},0\right)\to\left(\mathbb{C}^{p},0\right),\ \pi:\left(x,t\right)\to t$
such that $\left.\mathbb{F}\right|_{\pi^{-1}\left(0\right)}\sim\mathcal{F}$.
The \emph{equisingularity} property is a quite technical property
to state. However, it means basically that the topology of the process
of desingularization of the family of foliation$\left.\mathbb{F}\right|_{t=\alpha}$
does not depend on $\alpha.$ For the details, we refer to \cite{univ}. 
\begin{proof}
\textbf{Step 1 - }Let us prove that the deformation $\left(\mathcal{F}_{\sigma,\alpha}\right)_{\alpha\in\left(\mathbb{C},\alpha_{0}\right)}$
of $\mathcal{F}_{\sigma,\alpha_{0}}$ is induced by an unfolding.
We can make the following \emph{thick }gluing 
\[
\mathbb{F}:=\EQ{\mathcal{F}_{1}\times\left(\mathbb{C},\alpha_{0}\right)\coprod\mathcal{F}_{2}\times\left(\mathbb{C},\alpha_{0}\right)}{\left(x_{1},y_{1},\alpha\right)\to\left(\left(g_{\sigma}\circ\Phi_{\alpha_{0}}\right)\circ\left(\Phi_{\alpha_{0}}^{-1}\circ\Phi_{\alpha}\right)\left(x_{1},y_{1}\right),\alpha\right).}
\]
where $\mathcal{F}_{i}\times\left(\mathbb{C},\alpha_{0}\right)$ stands
for the product foliation: its leaves are the product of a leaf of
$\mathcal{F}_{i}$ and of an open neighborhood of $\alpha_{0}$ in
$\mathbb{C}$. The codimension $1-$foliation $\mathbb{F}$ comes
clearly with a fibration defined by the quotient of the map $\pi:\left(p,\alpha\right)\to\alpha$
whose fibers are transverse to the foliation. Thus, the above gluing
is an unfolding. Now, the restriction $\left.\mathbb{F}\right|_{\pi^{-1}\left(\alpha_{0}\right)}=\EQ{\mathcal{F}_{1}\coprod\mathcal{F}_{2}}{\left(g_{\sigma}\circ\Phi_{\alpha_{0}}\right)}$
is equal to $\mathcal{F}_{\sigma,\alpha_{0}}$. Finally, it is equisingular
by construction. Thus, it satisfies all the properties of an equisingular
unfolding in the sense of Mattei. 

\textbf{Step 2 - }Let us consider the sheaf $\Theta$ whose base is
the exceptionnal divisor $E^{-1}\left(0\right)=D=D_{2}\cup D_{1}$
of tangent vector fields to the foliation $E^{*}\mathcal{F}_{\sigma,\alpha_{0}}$
and to the divisor $E^{-1}\left(0\right)$. The cohomological group
$H^{1}\left(D,\Theta\right)$ represents the finite dimensionnal $\mathbb{C}$-space
of infinitesimal unfoldings. Following \cite{univ}, there exists
a\emph{ Kodaira-Spencer map like} that associate to any unfolding
with parameter in $\left(\mathbb{C}^{p},0\right)$, its Kodaira Spencer
derivative which is a linear map from $\mathbb{C}^{P}$ to $H^{1}\left(D,\Theta\right)$.
The unfolding is semi-universal as in the theorem above if and only
if its Kodaira Spencer derivative is a linear isomorphism. 

We consider the covering of the exceptionnal divisor $E^{-1}\left(0\right)$
by two open sets $U_{1}$ and $U_{2}$ where $U_{1}$ and $U_{2}$
are respectively tubular neighborhood of $D_{1}$ and $D_{2}$. It
is known that this covering is acyclic with respect to the sheaf $\Theta$,
i.e, $H^{1}\left(D_{i},\Theta\right)=0$. Therefore, following \cite{Godement}
to compute the cohomological group $H^{1}\left(D,\Theta\right)$ we
can use this covering, that is to say, the following isomorphism 
\begin{equation}
H^{1}\left(D,\Theta\right)\simeq\frac{H^{0}\left(U_{1}\cap U_{2},\Theta\right)}{H^{0}\left(U_{1},\Theta\right)\oplus H^{0}\left(U_{2},\Theta\right)}.\label{eq:iso}
\end{equation}
In view of the glued construction of $\mathcal{F}_{\sigma,\alpha_{0}}$,
a $0-$cocycle $X_{12}$ in $H^{0}\left(U_{1}\cap U_{2},\Theta\right)$
is trivial in $H^{1}\left(D,\Theta\right)$ if and only if the cohomological
equation 
\begin{equation}
X_{12}=X_{1}-\left(g_{\sigma}\circ\Phi_{\alpha_{0}}\right)^{*}X_{2}\label{eq:coho}
\end{equation}
admits a solution where $X_{1}\in H^{0}\left(U_{1},\Theta\right)$
and $X_{2}\in H^{0}\left(U_{2},\Theta\right)$. Now, it is known \cite{univ}\cite{Caco}
that the dimension of the $\mathbb{C}$ space$H^{1}\left(D,\Theta\right)$
is $1$. Thus, to prove the result, it is enough to show that the
image of the deformation $\left(\mathcal{F}_{\sigma,\alpha}\right)_{\alpha\in\left(\alpha_{0},\mathbb{C}\right)}$
by the foliated Kodaira-Spencer map is not trivial in $H^{1}\left(D,\Theta\right)$.
The foliation $\mathcal{F}_{\sigma,\alpha}$ is obtained from $\mathcal{F}_{\sigma,\alpha_{0}}$
by gluing with the automorphism 
\[
\Phi_{\alpha_{0}}^{-1}\circ\Phi_{\alpha}\left(x_{1},y_{1}\right)=\left(x_{1}\frac{1+\alpha y_{1}}{1+\alpha_{0}y_{1}},y_{1}\right).
\]
Thus, its image by the Kodaira-Spencer map is the cocycle 
\[
\left.\frac{\partial}{\partial\alpha}\Phi_{\alpha_{0}}^{-1}\circ\Phi_{\alpha}\right|_{\alpha=\alpha_{0}}=\frac{x_{1}y_{1}}{1+\alpha_{0}y_{1}}\frac{\partial}{\partial x_{1}}
\]
Hence, the unfolding is semi-universal if and only if the equation
\begin{equation}
x_{1}y_{1}\frac{\partial}{\partial x_{1}}+\cdots=X_{2}-\left(g_{\sigma}\circ\Phi_{\alpha_{0}}\right)^{*}X_{1}\label{eq:cohom}
\end{equation}
 has no solution. This equation can be more precisely written in the
following way 
\[
x_{1}y_{1}\frac{\partial}{\partial y_{1}}+\cdots=A_{2}\left(x_{1},y_{1}\right)x_{1}\frac{\partial}{\partial x_{1}}-\left(g_{\sigma}\circ\Phi_{\alpha_{0}}\right)^{*}\left(A_{1}\left(x_{3},y_{3}\right)x_{3}\frac{\partial}{\partial x_{3}}\right)
\]
where $A_{1}$ and $A_{2}$ are functions defined respectively in
$U_{1}$ and $U_{2}$. Let us write the Taylor expansion of $A_{2}=\sum_{ij}a_{ij}^{2}x_{1}^{i}y_{1}^{j}$.
In the coordinates $\left(x_{2},y_{2}\right)$ the function $A_{2}$
is written $A_{2}=\sum_{ij}a_{ij}^{2}x_{2}^{2i-j}y_{2}^{i}.$ Therefore,
if $a_{ij}^{2}\neq0$ then $2i-j\ge0$ and the monomial term $y_{1}$
cannot appear in the Taylor expansion of $A_{2}$. In the same way,
the Taylor expansion of $A_{1}=\sum_{ij}a_{ij}^{1}x_{3}^{i}y_{3}^{j}$,
satisfies $a_{ij}^{1}\neq0\Rightarrow i\geq j$. Since $X_{1}$ vanishes
along the exceptionnal divisor whose trace in $U_{1}$ is the diagonal
$x_{3}=y_{3}$, we have $A_{1}=\left(x_{3}-y_{3}\right)\tilde{A}_{1}$.
Thus, in the coordinates $\left(x_{1},y_{1}\right)$, $X_{1}$ is
written 
\[
X_{1}=\tilde{A_{1}}\left(x_{1}\left(1+\alpha_{0}y_{1}\right)+\sigma\left(y_{1}\right),\sigma\left(y_{1}\right)\right)\left(x_{1}\left(1+\alpha_{0}y_{1}\right)+\sigma\left(y_{1}\right)\right)x_{1}\frac{\partial}{\partial x_{1}}.
\]
If $\tilde{A}_{1}\left(0,0\right)=0$ then the term $y_{1}x_{1}\frac{\partial}{\partial y_{1}}$
of the Taylor expansion of the cocycle (\ref{eq:cohom}) cannot come
from $X_{1}$. However, if $\tilde{A_{1}}\left(0,0\right)\neq0$ then
$A_{1}$ cannot be global. Therefore, the equation (\ref{eq:cohom})
cannot be solved, which proves the result. 
\end{proof}
We observe that $\mathcal{F}_{\sigma,\alpha}$ is an unfolding over
the whole $\mathbb{C}.$ Actually in the course of the above proof,
we obtain a more precise result
\begin{cor}
\label{cor:More-generally,-for}More generally, for any germ of function
$A\left(x,y\right)$ with $A\left(0,0\right)\neq0$, the $\mathbb{C}$-space
$H^{1}\left(D,\Theta\right)$ for the foliation 
\[
\EQ{\mathcal{F}_{1}\coprod\mathcal{F}_{2}}{\left(x_{1},y_{1}\right)\to\left(x_{1}A\left(x_{1},y_{1}\right)+\sigma\left(y_{1}\right),\sigma\left(y_{1}\right)\right).}
\]
 is generated by the cocycle the image of $x_{1}y_{1}\frac{\partial}{\partial y_{1}}$
through the isomorphism (\ref{eq:iso}). In particular, any deformation
of the form 
\[
\epsilon\to\EQ{\left(\mathcal{F}_{1}\coprod\mathcal{F}_{2}\right)_{\epsilon}}{\left(x_{1},y_{1}\right)\to\left(x_{1}A_{\epsilon}\left(x_{1},y_{1}\right)+\sigma\left(y_{1}\right),\sigma\left(y_{1}\right)\right)}
\]
where $\frac{\partial A_{\epsilon}}{\partial y_{1}}\left(0,0\right)$
does not depend on $\epsilon$ is locally analytically trivial. 
\end{cor}
As an easy consequence of the corollary, we obtain a theorem of normalization
of the construction of absolutely dicritical foliations of cusp type.
\begin{thm}
Any absolutely dicritical foliation of cusp type is equivalent to
some $\mathcal{F}_{\sigma,\alpha}.$\end{thm}
\begin{proof}
Let us consider $\mathcal{F}$ an absolutely dicritical foliation
of cusp type and let $E$ be its associated reduction. Since along
each component of the exceptionnal divisor the foliation is purely
radial, there exists two automorphisms $\Phi_{1}$ and $\Phi_{2}$
that conjugates $\mathcal{F}$ respectively to the models $\mathcal{F}_{1}$
and $\mathcal{F}_{2}$ in the neighborhood of respectively $D_{1}$
and $D_{2}$. The cocycle of gluing is thus written $\Phi_{1}\circ\Phi_{2}^{-1}.$
Applying if necessary a global automorphism of $\Phi_{1}$ that let
invariant each leaf, we can suppose that $\Phi_{1}\circ\Phi_{2}^{-1}$
send the exceptionnal divisor $x_{1}=0$ on the line $x_{3}=y_{3}.$
Since the cocycle conjugates the foliations $\mathcal{F}_{2}$ and
$\mathcal{F}_{1}$, it can be written 
\[
\left(x_{1},y_{1}\right)\mapsto\left(x_{1}A\left(x_{1},y_{1}\right)+\sigma\left(y_{1}\right),\sigma\left(y_{1}\right)\right).
\]
for some $\sigma\in\mbox{Diff}\left(\mathbb{C},0\right)$ and some
$A\in\mathbb{C}\left\{ x_{1},y_{1}\right\} $ with $A\left(0,0\right)\neq0$.
Applying if necessary an automorphism of $\mathcal{F}_{2}$ defined
by $\left(\epsilon x_{3},\epsilon y_{3}\right)$ for some $\epsilon\neq0$,
we can suppose that $A\left(0,0\right)=1$. Now we can write the cocycle
\[
\left(x_{1},y_{1}\right)\mapsto\left(x_{1}\left(1+\alpha y_{1}+\tilde{A}\left(x_{1},y_{1}\right)\right)+\sigma\left(y_{1}\right),\sigma\left(y_{1}\right)\right).
\]
where no term of the form $ay_{1}$ appears in $\tilde{A}.$ According
to the corollary, the deformation parametrized by $\epsilon$ and
defined by the gluing cocycle 
\[
\left(x_{1},y_{1}\right)\mapsto\left(x_{1}\left(1+\alpha y_{1}+\epsilon\tilde{A}\left(x_{1},y_{1}\right)\right)+\sigma\left(y_{1}\right),\sigma\left(y_{1}\right)\right)
\]
is locally analytically trivial. Thus the foliation obtained setting
$\epsilon=1$ and $\epsilon=0$ are analytically equivalent and setting
$\epsilon=0$ yields a cocycle of the desired form. 
\end{proof}
The couple $\left(\alpha,\sigma\right)$ is unique up to conjugacies
fixing any point of the exceptionnal divisor. However, once we authorize
any kind of conjugacies, this couple is not unique anymore. But the
ambiguity can be described. 
\begin{prop}
\label{prop:Two-normal-forms}Two normal forms $\mathcal{F}_{\sigma,\alpha}$
and $\mathcal{F}_{\gamma,\alpha^{'}}$ are conjugated if and only
if there are two homographies $h_{0}$ and $h_{1}$ such that 
\begin{equation}
\begin{cases}
\sigma=h_{1}\circ\gamma\circ h_{0}\\
\frac{2}{5}\left(\alpha-\frac{3}{2}\frac{\sigma^{''}\left(0\right)}{\sigma^{'}\left(0\right)}\right)=\frac{2}{5}\left(\alpha^{'}-\frac{3}{2}\frac{\gamma^{''}\left(0\right)}{\gamma^{'}\left(0\right)}\right)h_{0}^{'}\left(0\right)-\frac{h_{0}^{''}\left(0\right)}{h_{0}^{'}\left(0\right)}
\end{cases}\label{eq:relation-fond}
\end{equation}
\end{prop}
\begin{proof}
\textbf{Step 1 -} In view of our gluing construction and following
the key remark (\ref{Key-remark.-Two}), the existence of a conjugacy
implies that there exist two automorphisms of respectively $\mathcal{F}_{2}$
and $\mathcal{F}_{1}$ written $\Phi_{2}=\left(x_{1}A_{2}\left(x_{1},y_{1}\right),h_{0}\left(y_{1}\right)\right)$
and $\Phi_{1}=\left(x_{3}A_{1}\left(x_{3},y_{3}\right),h_{1}\left(y_{3}\right)\right)$
such that 
\[
\left(x_{1}\left(1+\alpha y_{1}\right)+\sigma\left(y_{1}\right),\sigma\left(y_{1}\right)\right)=\Phi_{1}\circ\left(x_{1}\left(1+\alpha^{'}y_{1}\right)+\gamma\left(y_{1}\right),\gamma\left(y_{1}\right)\right)\circ\Phi_{2}.
\]
First, we obviously get the following relation $\sigma=h_{1}\circ\gamma\circ h_{0}.$
Moreover, if we look at the first component of the above relation
we get
\begin{eqnarray*}
x_{1}\left(1+\alpha y_{1}\right)+\sigma\left(y_{1}\right) & = & \left(x_{1}A_{2}\left(x_{1},y_{1}\right)\left(1+\alpha^{'}h_{0}\right)+\gamma\circ h_{0}\right)\times\\
 &  & A_{1}\left(x_{1}A_{2}\left(x_{1},y_{1}\right)\left(1+\alpha^{'}h_{0}\right)+\gamma\circ h_{0},\gamma\circ h_{0}\right)
\end{eqnarray*}
If we compute the derivative $\frac{\partial}{\partial x_{1}}$ of
the above relation and then set $x_{1}=0$, we get 
\begin{eqnarray}
1+\alpha y_{1} & = & A_{2}\left(0,y_{1}\right)\left(1+\alpha^{'}h_{0}\right)\times\nonumber \\
 &  & \left(\gamma\circ h_{0}\frac{\partial A_{1}}{\partial x_{1}}\left(\gamma\circ h_{0},\gamma\circ h_{0}\right)+A_{1}\left(\gamma\circ h_{0},\gamma\circ h_{0}\right)\right)\label{eq:alpha}
\end{eqnarray}

\begin{enumerate}
\item Now, since $\Phi_{1}$ preserve the curve $y=x,$ we obtain 
\[
A_{1}\left(x,x\right)=\frac{h_{1}\left(x\right)}{x}
\]
Thus, $A_{1}\left(0,0\right)=h_{1}^{'}\left(0\right).$ Setting $y_{1}=0$
in the relation above, we get $1=A_{2}\left(0,0\right)A_{1}\left(0,0\right).$
Therefore, $A_{2}\left(0,0\right)=\frac{1}{h_{1}^{'}\left(0\right)}.$
Now, let us write the Taylor expansion of $A_{1}$ 
\[
A_{1}\left(x_{3},y_{3}\right)=h_{1}^{'}\left(0\right)+rx_{3}+sy_{3}+\cdots.
\]
Since, $A_{1}\left(x,x\right)=\frac{h_{1}\left(x\right)}{x},$ we
have $r+s=\frac{h_{1}^{''}\left(0\right)}{2}$. Now, the biholomorphism
$\left(x_{3}A_{1}\left(x_{3},y_{3}\right),h_{1}\left(y_{3}\right)\right)$
is global: therefore, it can be push down and extended at the origin
of $\mathbb{C}^{2}$ as a local automorphism written 
\[
\left(x,y\right)\mapsto\left(xA_{1}\left(x,\frac{y}{x}\right),h_{1}\left(\frac{y}{x}\right)xA_{1}\left(x,\frac{y}{x}\right)\right).
\]
The second component of this expression is written 
\[
\frac{y}{\alpha x+\beta y}\left(h_{1}^{'}\left(0\right)x+rx^{2}+sy+\cdots\right)
\]
where $\alpha=\frac{1}{h_{1}^{'}\left(0\right)}$ and $\beta=-\frac{h_{1}^{''}\left(0\right)}{2h_{1}^{'}\left(0\right)^{2}}$.
It is extendable at $\left(0,0\right)$ if and only if the expression
in parenthesis can be holomorphically divided by $\alpha x+\beta y$.
Looking at the first jet of these expressions leads to 
\[
\left|\begin{array}{cc}
\beta & \alpha\\
s & h_{1}^{'}\left(0\right)
\end{array}\right|=0\Longrightarrow s=\frac{\beta h_{1}^{'}\left(0\right)}{\alpha}=-\frac{h_{1}^{''}\left(0\right)}{2}
\]
Finally, we have $r=h_{1}^{''}\left(0\right).$
\item In the same way, let us write the Taylor expansion of $A_{2}\left(x_{1},y_{1}\right)=\frac{1}{h_{1}^{'}\left(0\right)}+uy_{1}+vy_{1}^{2}+\cdots.$
The second component of the expression of $\Phi_{2}$ in the coordinates
$\left(x_{2},y_{2}\right)$ is $y_{2}x_{2}^{2}h_{0}^{2}\left(\frac{1}{x_{2}}\right)A_{2}\left(y_{2}x_{2}^{2},\frac{1}{x_{2}}\right)$
which is equal to 
\[
\frac{y_{2}}{\left(\alpha^{'}x_{2}+\beta^{'}\right)^{2}}\left(\alpha x_{2}^{2}+ux_{2}+v+y_{2}\left(\cdots\right)\right)
\]
where $\alpha^{'}=\frac{1}{h_{0}^{'}\left(0\right)}$ and $\beta^{'}=-\frac{h_{0}^{''}\left(0\right)}{2h_{0}^{'}\left(0\right)^{2}}$.
Since it is extendable at $x_{1}=-\frac{\beta^{'}}{\alpha^{'}}$,
there exists a constant $\Gamma$ such that $\left(\alpha^{'}x_{2}+\beta^{'}\right)^{2}=\Gamma\left(\alpha x_{2}^{2}+ux_{2}+v\right)$
Hence, we have the equality $u=2\frac{\alpha\beta^{'}}{\alpha^{'}}=-\frac{h_{0}^{''}\left(0\right)}{h_{0}^{'}\left(0\right)h_{1}^{'}\left(0\right)}$.
\end{enumerate}
Now, we can identified the coefficient of the equation (\ref{eq:alpha})

It is 
\begin{eqnarray*}
\alpha & = & A_{2}\left(0,0\right)\left(\gamma^{'}\left(0\right)h_{0}^{'}\left(0\right)\frac{\partial A_{1}}{\partial x_{1}}\left(0,0\right)+\frac{h_{1}^{''}\left(0\right)}{2}\gamma^{'}\left(0\right)h_{0}^{'}\left(0\right)+\alpha^{'}h_{0}^{'}\left(0\right)h_{1}^{'}\left(0\right)\right)\\
 &  & \quad+uh_{1}^{'}\left(0\right)\\
 & = & \frac{3}{2}\gamma^{'}\left(0\right)h_{0}^{'}\left(0\right)\frac{h_{1}^{''}\left(0\right)}{h_{1}^{'}\left(0\right)}-\frac{h_{0}^{''}\left(0\right)}{h_{0}^{'}\left(0\right)}+\alpha^{'}h_{0}^{'}\left(0\right).
\end{eqnarray*}
Using the relation $\sigma=h_{1}\circ\gamma\circ h_{0}$ the above
equality can be formulated as in the theorem. 

\textbf{Step 2 - }We suppose that the conclusion of the statement
is satisfied. Let us suppose that 
\[
h_{1}\left(z\right)=\frac{z}{\alpha+\beta z}\qquad h_{0}\left(z\right)=\frac{z}{a+bz}
\]
Then we set
\begin{eqnarray*}
A_{2}\left(x_{1},y_{1}\right) & = & \alpha+2\frac{\alpha b}{a}y_{1}+\frac{\alpha b}{a^{2}}y_{1}^{2}\\
A_{1}\left(x_{3},y_{3}\right) & = & \frac{\alpha+\beta y_{3}}{\left(\alpha+\beta x_{3}\right)^{2}}.
\end{eqnarray*}
In view of the computations done in the first step, the two automorphisms
$\Phi_{1}$ and $\Phi_{2}$ associated to $A_{1}$ and $A_{2}$ can
be extended on $U_{1}$ and $U_{2}$, tubular beighborhood of $D_{1}$
and $D_{2}$. Moreover, we obtain the following relation 
\begin{eqnarray*}
\left(x_{1}\left(1+\alpha y_{1}+\Delta\left(x_{1},y_{1}\right)\right)+\sigma\left(y_{1}\right),\sigma\left(y_{1}\right)\right)\qquad\\
=\Phi_{1}\circ\left(x_{1}\left(1+\alpha^{'}y_{1}\right)+\gamma\left(y_{1}\right),\gamma\left(y_{1}\right)\right)\circ\Phi_{2}
\end{eqnarray*}
where $\Delta$ does not contain any monomial term in $y_{1}$. Now,
using the proposition (\ref{cor:More-generally,-for}), we see that
the deformation defined by 
\[
\epsilon\to\left(x_{1}\left(1+\alpha y_{1}+\epsilon\Delta\left(x_{1},y_{1}\right)\right)+\sigma\left(y_{1}\right),\sigma\left(y_{1}\right)\right)
\]
is analytically trivial, which ensures the theorem. \end{proof}
\begin{thm}
The moduli space of absolutely dicritical foliations of cusp type
can be identified with the functionnal space $\mathbb{C}\left\{ z\right\} $
up to the action of $\mathbb{C}^{*}$ defined by 
\[
\epsilon\cdot\left(z\mapsto\sigma\left(z\right)\right)=\epsilon^{2}\sigma\left(\epsilon z\right).
\]
\end{thm}
\begin{proof}
We can consider the following family parametrized by $\mbox{Diff}\left(\mathbb{C},0\right)$
\[
\sigma\in\mbox{Diff}\left(\mathbb{C},0\right)\to\mathcal{F}_{\frac{3}{2}\frac{\sigma^{''}\left(0\right)}{\sigma^{'}\left(0\right)},\sigma}.
\]
It is a complete family for absolutely dicritical foliations of cusp
type: in any class of absolutely dicritical foliation of cusp type
there is one that is analytically equivalent to one of the form $\mathcal{F}_{\frac{3}{2}\frac{\sigma^{''}\left(0\right)}{\sigma^{'}\left(0\right)},\sigma}$.
Indeed, considering the foliation $\mathcal{F}_{\alpha^{'},\gamma}$,
we can choose $h_{0}$ such that $\frac{2}{5}\left(\alpha^{'}-\frac{3}{2}\frac{\gamma^{''}\left(0\right)}{\gamma^{'}\left(0\right)}\right)h_{0}^{'}\left(0\right)-\frac{h_{0}^{''}\left(0\right)}{h_{0}^{'}\left(0\right)}=0$.
Therefore, setting $\sigma=\gamma\circ h_{0}$ ensures that $\mathcal{F}_{\alpha^{'},\gamma}$
and $\mathcal{F}_{\frac{3}{2}\frac{\sigma^{''}\left(0\right)}{\sigma^{'}\left(0\right)},\sigma}$
are analytically equivalent. Moreover, if $\mathcal{F}_{\frac{3}{2}\frac{\sigma_{0}^{''}\left(0\right)}{\sigma_{0}^{'}\left(0\right)},\sigma_{0}}$
and $\mathcal{F}_{\frac{3}{2}\frac{\sigma_{1}^{''}\left(0\right)}{\sigma_{1}^{'}\left(0\right)},\sigma_{1}}$
are analytically equivalent then there exists $\epsilon\in\mathbb{C}^{*}$
and an homographie $h_{1}$ such that 
\begin{equation}
\sigma_{0}\left(z\right)=h_{1}\circ\sigma_{1}\circ\left(\epsilon z\right).\label{eq:rel-simp}
\end{equation}
Indeed, the second homographie $h_{0}$ that appears in the proposition
(\ref{prop:Two-normal-forms}) has to be linear for the relations
(\ref{eq:relation-fond}) ensures that $h_{0}^{''}\left(0\right)=0$.
Thus, $h_{0}$ is written $z\mapsto\epsilon z$ for some $\epsilon.$
To simplify the relation (\ref{eq:rel-simp}), we use the Schwartzian
derivative which is a surjective operator defined by 

\[
\mathcal{S}:\begin{cases}
\mbox{Diff}\left(\mathbb{C},0\right)\to\mathbb{C}\left\{ z\right\} \\
y\mapsto\frac{3}{2}\left(\frac{y^{'''}}{y^{'}}\right)-\left(\frac{y^{''}}{y^{'}}\right)^{2}
\end{cases}
\]
and satisfying the following property: the relation (\ref{eq:rel-simp})
is equivalent to $\mathcal{S}\left(\sigma_{0}\right)\left(z\right)=\epsilon^{2}\mathcal{S}\left(\sigma_{1}\right)\left(\epsilon z\right)$.
Therefore, the moduli space of absolutely dicritical foliation of
cusp type is identified via the Schwartzian derivative to the quotient
of $\mathbb{C}\left\{ z\right\} $ up to the action of $\mathbb{C}^{*}$$\epsilon\cdot\left(z\mapsto\sigma\left(z\right)\right)=\epsilon^{2}\sigma\left(\epsilon z\right).$
\end{proof}
As mentionned in the introduction, this theorem does not state that
the transversal structure $\sigma$ is the sole analytical invariant
of an absolutely dicritical foliation of cusp type. Indeed, the action
of the group of conjugacies act transversaly to the transverse structures
$\sigma$ and to the moduli of Mattei $\alpha$. The family $\mathcal{F}_{\frac{3}{2}\frac{\sigma^{''}\left(0\right)}{\sigma^{'}\left(0\right)},\sigma}$
is a complete tranversal set for this action 

\noindent \begin{center}
\begin{picture}(0,0)%
\includegraphics{conj.pstex}%
\end{picture}%
\setlength{\unitlength}{4144sp}%
\begingroup\makeatletter\ifx\SetFigFont\undefined%
\gdef\SetFigFont#1#2#3#4#5{%
  \reset@font\fontsize{#1}{#2pt}%
  \fontfamily{#3}\fontseries{#4}\fontshape{#5}%
  \selectfont}%
\fi\endgroup%
\begin{picture}(2234,1503)(1246,-2914)
\put(2296,-1636){\makebox(0,0)[lb]{\smash{{\SetFigFont{6}{7.2}{\rmdefault}{\mddefault}{\updefault}{\color[rgb]{0,0,0}the group of conjugacies.}%
}}}}
\put(2296,-1501){\makebox(0,0)[lb]{\smash{{\SetFigFont{6}{7.2}{\rmdefault}{\mddefault}{\updefault}{\color[rgb]{0,0,0}Orbits of the action of }%
}}}}
\put(1261,-1597){\makebox(0,0)[lb]{\smash{{\SetFigFont{6}{7.2}{\rmdefault}{\mddefault}{\updefault}{\color[rgb]{0,0,0}$\alpha\in\mathbb{C}$}%
}}}}
\put(2566,-2311){\makebox(0,0)[lb]{\smash{{\SetFigFont{6}{7.2}{\rmdefault}{\mddefault}{\updefault}{\color[rgb]{0,0,0}$\mathcal{F}_{\frac{3}{2}\frac{\sigma^{''}(0)}{\sigma^{'}(0)},\sigma}$}%
}}}}
\put(2223,-2872){\makebox(0,0)[lb]{\smash{{\SetFigFont{6}{7.2}{\rmdefault}{\mddefault}{\updefault}{\color[rgb]{0,0,0}$\sigma\in$Diff$(\mathbb{C},0)$}%
}}}}
\end{picture}%

\par\end{center}

As a consequence of the above description of the moduli space of absolutely
dicritical foliations, we should be able to prove the existence of
a non algebrizable absolutely dicritical foliation using technics
developped in \textbf{\cite{singnonag}. }

\section{Formal normal forms for $1$-Forms.}

It is known \cite{Caco} that the valuation of a $1$-form $\omega$
with an isolated singularity defining an absolutely dicritical foliation
of cusp type is $3$. Up to some linear change of coordinates, we
can suppose that the singular point of the foliation after one blowing-up
has $\left(0,0\right)$ for coordinates in the standard coordinates
associated to the blowing-up. Moreover, since the foliation is generically
transversal to the exceptionnal divisor of the blowing-up of $0\in\left(\mathbb{C}^{2},0\right)$,
the homogeneous part of degree $3$ of $\omega$ is tangent to the
radial form $\omega_{R}=xdy-ydx.$ Thus there exists an homogeneous
polynomial function of degree $2$ $P_{2}$ such that 

\[
\omega=P_{2}\omega_{R}+\sum_{i\geq4}\left(A_{i}\left(x,y\right)dx+B_{i}\left(x,y\right)dy\right).
\]
After one blowing-up, the singular locus is given by the solutions
of $P_{2}\left(1,y\right)=0$ and $P_{2}\left(x,1\right)=0$ in each
chart. Thus $P_{2}$ is simply written $ay^{2}$ for some constant
$a\neq0$. After on blowing-up $\left(x,t\right)\mapsto\left(x,tx\right)$,
the linear part near $\left(0,0\right)$ of the pull-back form is
written 
\[
\left(A_{4}\left(1,0\right)+t\frac{\partial A_{4}}{\partial t}\left(1,0\right)+tB_{4}\left(1,0\right)\right)dx+xB_{4}\left(1,0\right)dt+xA_{5}\left(1,0\right)dx.
\]
The absolutely dicritical property ensures that this linear part is
non trivial and tangent to the radial vector field $tdx+xdt$. Hence,
the following relations hold

\[
A_{4}\left(1,0\right)=A_{5}\left(1,0\right)=0\mbox{ and }\frac{\partial A_{4}}{\partial t}\left(1,0\right)+2B_{4}\left(1,0\right)=0
\]
Finally, the form $\omega$ is written 
\begin{eqnarray*}
\omega & = & y^{2}\omega_{R}+\left(-2\alpha x^{3}+yQ_{2}\left(x,y\right)\right)ydx+\left(\alpha x^{4}+yQ_{3}\left(x,y\right)\right)dy\\
 &  & \quad+\left(A_{5}\left(x,y\right)dx+B_{5}\left(x,y\right)dy\right)+\cdots
\end{eqnarray*}
where $\alpha\neq0$. 
\begin{prop}
The $1-$form $\omega$ is formally equivalent to a $1-$form written
\begin{eqnarray*}
y^{2}\omega_{R}+\alpha x^{3}\left(xdy-2ydx\right)+ax^{3}ydy\qquad\qquad\\
+\sum_{n\geq5}x^{n-1}\left(\left(a_{n}x+b_{n}y\right)dx+\left(c_{n}x+d_{n}y\right)dy\right)
\end{eqnarray*}
where $a_{5}=0$. Moreover, this formal normal form is unique up to
change of coordinates tangent to $\mbox{Id}$. \end{prop}
\begin{proof}
The action of a change of coordinates $\phi_{n}:\left(x,y\right)\to\left(x,y\right)+\left(P_{n},Q_{n}\right)$
where $P_{n}$ and $Q_{n}$ are homogeneous polynomial functions of
degree $n$ does not modify the jet of order $n+1$ of $\omega$.
Moreover, the action on the homogeneous part of degree $n+2$ is written
\begin{eqnarray*}
J^{n+2}\left(\phi_{_{n}}^{*}\omega\right) & = & J^{n+2}\omega\\
 &  & +y^{2}\left(\left(x\frac{\partial Q_{n}}{\partial x}-y\frac{\partial P_{n}}{\partial x}+Q_{n}\right)dx+\left(x\frac{\partial Q_{n}}{\partial y}-x\frac{\partial P_{n}}{\partial x}+P_{n}\right)dy\right)
\end{eqnarray*}
We are going to verify that the linear morphism defined by 
\[
L:\left(P_{n},Q_{n}\right)\mapsto\left(x\frac{\partial Q_{n}}{\partial x}-y\frac{\partial P_{n}}{\partial x}+Q_{n},x\frac{\partial Q_{n}}{\partial y}-x\frac{\partial P_{n}}{\partial x}+P_{n}\right)
\]
from the set of couples of homogeneous polynomial functions of degree
$n$ to itself is a one to one correspondance. To do so, let us compute
the kernel of this morphism and let us write $P_{n}=\sum_{i=0}^{n}p_{i}x^{i}y^{n-i}$
and \foreignlanguage{french}{$Q_{n}=\sum_{i=0}^{n}q_{i}x^{i}y^{n-i}$.}
The coefficients of the components of $L\left(P_{n},Q_{n}\right)$
on the monomial term $x^{i}y^{n-i}$ are 
\begin{eqnarray*}
q_{i}\left(i-1\right)-p_{i+1}\left(i+1\right)\quad i & = & 0..n-1\\
q_{n}\left(n-1\right)\quad i & = & n\mbox{ \qquad and}\\
-p_{i}\left(n-i-1\right)+q_{i-1}\left(n-i+1\right)\quad i & = & 1..n\\
p_{0}\left(n-1\right)\quad i & = & 0.
\end{eqnarray*}
If $\left(P_{n},Q_{n}\right)$ is in the kernel then $q_{n}=0$ and
$p_{0}=0$. Moreover, applying the above relation with $i=1$ and
$i=n-1$ yields $p_{2}=0$ and $q_{n-2}=0.$ Now for $i=1..n-1$ but
$i\neq n-2$, a combination of the relations above ensures that 
\begin{eqnarray*}
0=q_{i}\left(i-1\right)-q_{i}\left(i+1\right)\frac{n-i}{n-i-2} & = & \frac{q_{i}}{n-i-2}\left(2-2n\right)
\end{eqnarray*}
Thus $q_{i}=0$ for $i=0..n-1$. Therefore $\left(P_{n},Q_{n}\right)=0$
and $L$ is an isomorphism. Thus, we can choose $\phi_{n}$ such that
\[
J^{n+2}\left(\phi_{_{n}}^{*}\omega\right)=x^{n-1}\left(\left(a_{n}x+b_{n}y\right)dx+\left(c_{n}x+d_{n}y\right)dy\right).
\]
 Clearly the composition $\phi_{2}\circ\phi_{3}\circ\cdots$ is formally
convergent, which proves the proposition.
\end{proof}

\section{Absolutely dicritical foliation admitting a first integral.}

In this section, we study absolutely dicritical foliations that admit
a meromorphic first integral. Such an existence can be completely
red on the transverse structure. 
\begin{thm}
Let $\mathcal{F}$ be an absolutely dicritical foliation of cusp type
with $\sigma$ as transverse structure. Then $\mathcal{F}$ admits
a first integral if and only if there exists two non constant rationnal
functions $R_{1}$ and $R_{2}$ such that 
\[
R_{1}\circ\sigma=R_{2}.
\]

\end{thm}
Notice that the existence of $R_{1}$ and $R_{2}$ does not depend
on the equivalence class of $\sigma$ modulo homographies. 
\begin{proof}
Suppose that $\mathcal{F}$ admits a meromorphic first integral $f$.
After blowing-up, the function $f$ is a non constant rationnal function
in restriction to each component of the divisor. Since for any point
$p$, $p$ and $\sigma\left(p\right)$ belongs to the same leaf, we
have 
\[
\left.f\right|_{D_{1}}\left(p\right)=\left.f\right|_{D_{2}}\left(\sigma\left(p\right)\right).
\]
Now, suppose there exist two rationnal function as in the lemma. According
to some previous result, there exists $\alpha$ and $\gamma$ such
that the foliation $\mathcal{F}$ is analytically equivalent to $\mathcal{F}_{\alpha,\gamma}.$
The application $\sigma$ and $\gamma$ are linked by a relation of
the form 
\[
h_{0}\circ\sigma\circ h_{1}=\gamma
\]
where $h_{0}$ and $h_{1}$ are homographies. Thus, setting $\tilde{R}_{1}=R_{1}\circ h_{0}^{-1}$
and $\tilde{R}_{2}=R_{2}\circ h_{1}$ yields $\tilde{R}_{1}\circ\gamma=\tilde{R}_{2}$
where $\tilde{R}_{1}$ and $\tilde{R}_{2}$ are still rationnal. Now,
let us go back to the construction of $\mathcal{F}_{\alpha,\gamma}.$
We glue the models $\mathcal{F}_{1}$ and $\mathcal{F}_{2}$ around
$\left(x_{1},y_{1}\right)=0$ and $\left(x_{3},y_{3}\right)=0$ by
\[
\left(x_{1},y_{1}\right)\mapsto\left(x_{3}=x_{1}\left(1+\alpha y_{1}\right)+\gamma\left(y_{1}\right),y_{3}=\gamma\left(y_{1}\right)\right)
\]
Consider for $\mathcal{F}_{1}$ the first integral $F_{1}\left(x_{1},y_{1}\right)=\tilde{R}_{2}\left(y_{1}\right)$
and for $\mathcal{F}_{2}$ the first integral $F_{2}\left(x_{3},y_{3}\right)=\tilde{R}_{1}\left(y_{3}\right)$.
Then these two meromorphic first integrals can be glued in a global
meromorphic first integral since 
\[
F_{2}\left(x_{3},y_{3}\right)=F_{2}\left(x_{1}\left(1+\alpha y_{1}\right)+\gamma\left(y_{1}\right),\gamma\left(y_{1}\right)\right)=\tilde{R}_{1}\left(\gamma\left(y_{1}\right)\right)=\tilde{R}_{2}\left(y_{1}\right)=F_{1}\left(x_{1},y_{1}\right).
\]
Thus the absolutely dicritical foliation admits a meromorphic first
integral. 
\end{proof}
In view of this result, it is easy to produce a lot of examples of
absolutely dicritical foliation admitting no meromorphic first integral
setting for instance 
\[
\sigma\left(z\right)=e^{z}-1.
\]
Notice that the existence of the first integral depends only on the
transversal structure $\sigma$ and not on the value of the moduli
of Mattei $\alpha.$ This is consistent with the fact that along an
equireducible unfolding the existence of a meromorphic first integral
for one foliation in the deformation ensures the existence of such
a first integral for any foliation in the deformation. 

Finally, since the topologically classification of absolutely dicritical
foliations is \emph{trivial, }the above result produce a lot of examples
of couples of conjugated foliations such that only one of them admits
a meromorphic first integral. 

Hereafter we treated a special case, that is when the transversal
structure $\sigma$ is an homography.
\begin{prop}
\label{prop:Let--be}Let $\mathcal{F}$ be an absolutely dicritical
foliation of cusp type with an homographic transversal structure.
Then, up to some analytical change of coordinates, $\mathcal{F}$
admits one of the following rationnal first integrals:
\begin{enumerate}
\item $f=\frac{y^{2}+x^{3}}{xy}.$
\item $f=\frac{y^{2}+x^{3}}{xy}+x$
\end{enumerate}
\end{prop}
\begin{proof}
Let us consider the following germ of family of meromorphic functions
with $\left(x,y,z\right)\in\left(\mathbb{C}^{3},\left(0,0,0\right)\right)$
defined by 
\[
f_{z}=\frac{y^{2}+x^{3}+zx^{2}y}{xy}=\frac{a}{b}.
\]
For any $z$, the foliation associated to $f_{z}$ is absolutely dicritical
of cusp type. Let us prove that this family is an equireducible unfolding.
We consider the integrable $1-$form $\Omega=adb-bda$. It is written
\[
\left(2x^{3}y+zx^{2}y^{2}-y^{3}\right)dx+\left(xy^{2}-x^{4}\right)dy+x^{3}y^{2}dz.
\]
It defines an unfolding of the foliation given by $f_{0}$ with one
parameter. Its singular locus is the $z-$axes and it is transversal
to the fibers of the fibration $\left(x,y,z\right)\mapsto z$. Once
we blow-up the $z-$axe, in the chart $E:\left(x,t,z\right)=\left(x,tx,z\right)$,
the $1-$form $\Omega$ is written 
\[
\tilde{\Omega}=t\left(1-zt\right)dx+\left(t^{2}-x\right)dt+t^{2}xdz.
\]
Therefore, the singular locus of the pull-back foliation is still
the $z-$axe in the coordinates $\left(x,t,z\right)$ and in a neighborhood
of $x=0$ the foliation $\tilde{\Omega}$ is transverse to the fibration
$z=cst.$ If we blow-up again the $z$-axe we find
\[
\left(1-zx\right)dt+\left(1-zt\right)dx+txdz
\]
which is smooth. Since the curve $x=t=0$ is invariant and since the
foliation is still transverse to the fibration $z=cts,$ the unfolding
is equisingular. Now, this unfolding is analytically trivial if and
only if the monomial term $x^{3}y^{2}$ belongs to the ideal generated
by $2x^{3}y+zx^{2}y^{2}-y^{3}$ and $xy^{2}-x^{4}$ \cite{MC}. Setting
$z=0$ this would imply that $x^{3}y^{2}\in\left(2x^{3}y-y^{3},xy^{2}-x^{4}\right)$
which is impossible. Thus, this unfolding is not analytically trivial
and since the moduli space of unfolding of absolutely dicritical foliations
is of dimension $1$, it is also semi-universal.

Now, let us consider a foliation $\mathcal{F}$ as in the proposition.
Up to some linear change of coordinate, we can suppose that after
the reduction process its singular point and its transversal structure
are the same as the function $\frac{x^{2}+y^{3}}{xy}$ that is to
say $\left(0,0\right)$ and $\mbox{Id }$in the standard coordinates
associated to the reduction process. Let us denote by $\mathcal{F}_{0}$
the foliation given by $\frac{x^{2}+y^{3}}{xy}$. We are going to
construct an unfolding from $\mathcal{F}_{0}$ to $\mathcal{F}$.
As always since the beginning of this article, we denote by $D_{1}$
and $D_{2}$ the two exceptionnal component of the divisor. In the
neighborhood of each of them, both foliation are purely radial. Thus
there exists two conjugacy $\Phi_{1}$ and $\Phi_{2}$ defined in
the neighborhood of respectively $D_{1}$ and $D_{2}$ such that 
\[
\begin{array}{c}
\Phi_{1}^{*}\mathcal{F}_{0}=\mathcal{F}\qquad\Phi_{2}^{*}\mathcal{F}_{0}=\mathcal{F}\\
\left.\Phi_{1}\right|_{D_{1}\cup D_{2}}=\mbox{Id}\qquad\left.\Phi_{2}\right|_{D_{1}\cup D_{2}}=\mbox{Id}
\end{array}.
\]
Since, $\mathcal{F}_{0}$ and $\mathcal{F}$ have the same transversal
structures, the cocycle $\Phi_{1}\circ\Phi_{2}^{-1}$ is a germ automorphism
of $\mathcal{F}_{0}$ near the singular point of the divisor that
lets fix the points of the divisor and that let globally fix each
leaf. It is easy to see that one can construct an isotopy from $\Phi_{1}\circ\Phi_{2}^{-1}$
to $\mbox{Id}$ in the group of germs of automorphisms of $\mathcal{F}_{0}$
near the singular point of the divisor that let fix each point of
the divisor and that let globally fix each leaf. Let us denote by
$\Phi_{t}$ this isotopy satisfying $\Phi_{0}=\mbox{Id}$ and $\Phi_{1}=\Phi_{1}\circ\Phi_{2}^{-1}$.
The unfolding defined by the following glued construction 
\[
\EQ{\left(\left(\mathcal{F}_{0},D_{1}\right)\times U\right)\coprod\left(\left(\mathcal{F}_{0},D_{2}\right)\times U\right)}{\left(x,t\right)\sim\left(\Phi_{t}\left(x\right),t\right)}.
\]
where $U$ is an open neighborhood of $\left\{ \left|t\right|\leq1\right\} $
links $\mathcal{F}_{0}$ and $\mathcal{F}$. The meromophic first
integral $f_{0}$ of $\mathcal{F}_{0}$ can be extended in a meromorphic
first integral $F$ of the whole unfolding \cite{MC}. Thus $\left.F\right|_{t=1}$
is a meromorphic first integral of $\mathcal{F}$. By equisingularity$\left.F\right|_{t=0}$
and $\left.F\right|_{t=1}$ must have exactly the same number of irreducible
components in their zeros and in their poles, which is the same number
of irreducible components in the zeros and in the poles of $F$. They
also must have the same topology since an unfolding is topologically
trivial. Thus the foliation $\mathcal{F}$ admits a meromorphic first
integral whose zero is exactly the leaf passing through the singular
point of the exceptionnal divisor and whose poles are the union of
two smooth curves attaching respectively to $D_{1}$ and $D_{2}.$
Thererfore up to some change of coordinates, we can suppose that $\mathcal{F}$
has a meromorphic first integral of the form 

\[
f=\frac{\left(y^{2}+x^{3}+\Delta\left(x,y\right)\right)^{a}}{x^{b}y^{c}}
\]
where the Taylor expansion of $\Delta\left(x,y\right)$ admits monomial
term $x^{i}y^{j}$ with $2i+3j>6$. The absolutely dicritical property
ensures that $a=b=c.$ Therefore, we can suppose that $a=b=c=1.$
Let us denote by $\Lambda_{\lambda}\left(x,y\right)$ the homothetie
$\Lambda_{\lambda}\left(x,y\right)=\left(\lambda^{2}x,\lambda^{3}y\right)$.
Composing by $\Lambda_{\lambda}$ at the right of $f$ yields 
\[
\frac{f\circ\Lambda_{\lambda}}{\lambda}=\frac{y^{2}+x^{3}+\Delta_{\lambda}\left(x,y\right)}{xy}
\]
For any $\lambda\neq0$, the foliation given by $f$ and by $\frac{f\circ\Lambda_{\lambda}}{\lambda}$
are analytically conjugated. But the deformation given by $\lambda\to\frac{f\circ\Lambda_{\lambda}}{\lambda}$
is an equisingular unfolding of $f_{0}$ since $\Delta_{\lambda}$
goes to $0$ when $\lambda\to0$. Using the semi-universality of the
family introduced at the beginning of the proof, for $\lambda$ small
enough, there exists some $\alpha$ such that the following conjugacies
holds 
\[
f\sim\frac{f\circ\Lambda_{\lambda}}{\lambda}\sim f_{\alpha}.
\]
Now if $\alpha=0$ then $f$ is of type $\left(1\right)$. If $\alpha\neq0$,
applying some well-chosen homothetie, we can suppose $\alpha=1$.
And $f$ is of type $\left(2\right)$. \end{proof}
\begin{rem}
In the last part of this article, we will prove that actually the
two functions $\left(1\right)$ and $\left(2\right)$ of the previous
result define two foliations analytically equivalent. 
\end{rem}
It is possible to construct some others examples of absolutely dicritical
foliations of cusp type with a rationnal first integral: to do so,
consider a foliation of degree $1$ on $\mathbb{P}^{2}$. These are
well-known \cite{CerDeBel}: they have three singular points counted
with multiplicities and admit an integrating factor. For instance,
the foliation given in homogeneous coordinates by the multivalued
functions 
\[
\left[x:y:z\right]\to\frac{x^{\alpha}y^{\beta}}{z^{\alpha+\beta}}\quad\mbox{ or }\left[x:y:z\right]\mapsto\frac{Q}{z^{2}}
\]
where $Q$ is a non-degenerate quadratic form is of degree $1.$ When
$\alpha$ and $\beta$ are rationnal numbers, the foliation admits
a rationnal first integral. Now consider two generic lines $L_{1}$
and $L_{2}$. Each of them is tangent to one leaf of the foliation.
We can suppose that the tangency point is different from the intersection
point of $L_{1}$ and $L_{2}.$ Now, blow-up twice the tangency point
on $L_{1}$ and thrice the tangency point on $L_{2}$. The final configuration
is the following 

\noindent \begin{center}
\includegraphics[scale=0.4]{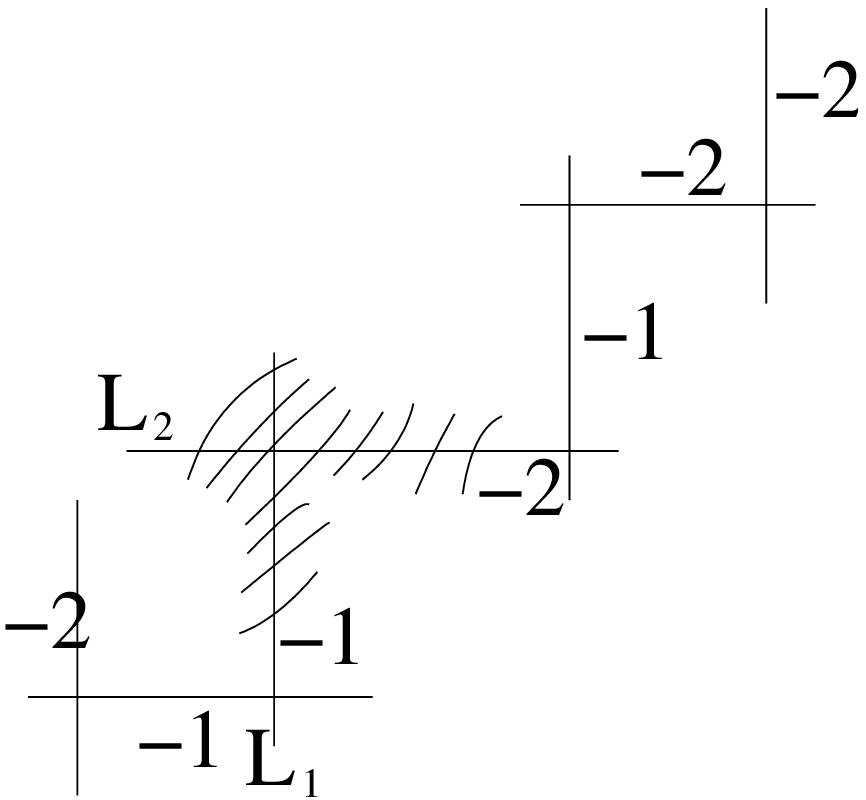}
\par\end{center}

Thus, the divisor $L_{1}\cup L_{2}$ can be contracted toward a smooth
algebraic manifold. The obtained singularity is naturally absolutely
dicritical of cusp type and admits a rationnal first integral. For
instance, if we consider the foliation given in affine coordinates
by $xy=cst$ and $L_{1}:x+y=1$ and $L_{2}:x-y=1$, the transverse
structure is equivalent to $\sigma\left(t\right)=t+1$ and thus the
foliation is equivalent to the functions of proposition (\ref{prop:Let--be}).
However, considering the foliation given by $x+y^{2}=cst$ yields
the transverse structure $t\mapsto\frac{1-\sqrt{1+12t+4t^{2}}}{2}$
which is not an homography.

\section{Moduli of Mattei.}

\subsection{The parameter space of the unfoldings.}

As already explain, the deformation $\alpha\to\mathcal{F}_{\alpha,\sigma}$
is an unfolding with a set of paramater equal to $\mathbb{C}$. It
is a natural problem to ask if two parameters define two foliations
analytically equivalent. In order to do so, we introduced the following
definition:
\begin{defn}
Let $\sigma$ be an element of $\mbox{Diff}\left(\mathbb{C},0\right)$.
An homography $h$ with $h\left(0\right)=0$ is called an homographic
symetry of $f$ if and only if there exists an homography $h_{1}$
such that 
\begin{equation}
h_{1}\circ\sigma\circ h=\sigma.\label{eq:Homo}
\end{equation}
We denote by $\mathcal{H}\left(f\right)$ the group of homographic
symetries of $f$. 
\end{defn}
The following result is probably known but we do not find any reference
in the litterature.
\begin{lem}
If $\mathcal{H}\left(f\right)$ is infinite then $f$ is an homography
and $\mathcal{H}\left(f\right)$ is the whole set of homographies
fixing the origin.\end{lem}
\begin{proof}
The relation (\ref{eq:Homo}) is equivalent to the functionnal equation
\[
f\circ h\left(z\right)=\frac{1}{\left(h^{'}\right)^{2}}f\left(z\right)
\]
where $f=S\left(\sigma\right)$ is the Schwartzian derivative of $\sigma$.
Let us write $h\left(z\right)=\frac{z}{a+bz}$ and $f\left(z\right)=\sum_{n\geq1}f_{n}z^{n}.$
\begin{enumerate}
\item Suppose that $h^{'}\left(0\right)$ is not a root of unity. Then applying
the above relation at $z=0$ leads to $f\left(0\right)=0$. Now, we
have 
\begin{eqnarray*}
a^{2}\sum_{n\geq1}f_{n}\frac{z^{n}}{\left(a+bz\right)^{n}} & = & \left(a+bz\right)^{4}\sum_{n\geq1}f_{n}z^{n}.
\end{eqnarray*}
An easy induction on $n$ show that for any $n$ $f_{n}=0$, thus
$f=0$ and $\sigma$ is an homography. 
\item Suppose now $h^{'}\left(0\right)=1$ then
\begin{eqnarray*}
\sum_{n\geq0}f_{n}\frac{z^{n}}{\left(1+bz\right)^{n}} & = & \left(1+bz\right)^{4}\sum_{n\geq0}f_{n}z^{n}
\end{eqnarray*}
Suppose that $b\neq0.$ If for any $n\leq N-1$ we have $f_{n}=0$,
let us have a look at the terms in $x^{N+1}$ in the above equality.
It is 
\[
-Nbf_{N}+f_{N+1}=4bf_{N}+f_{N+1}
\]
Thus $f_{N}=0.$ Which, proves by induction that $f$ still is equal
to zero. 
\item If $\mathcal{H}\left(f\right)$ is infinite, suppose it admits two
elements $h$ and $g$ that did not commute, then $\left[h,g\right]$
is tangent to $Id$ but is not the $Id$ . Thus using the above computation,
$f=0$. 
\item Finally, if $h^{'}\left(0\right)$ is a root of unity, it is easly
seen that $h^{\circ\left(n\right)}=\mbox{Id}$ where $n$ is the smallest
integer such that $h^{'}\left(0\right)^{n}=1.$ Thus, suppose that
the group $\mathcal{H}\left(f\right)$ is abelian and any element
of finite order. We have an embedding 
\[
\mathcal{H}\left(f\right)\longrightarrow\mbox{Aff}\left(\mathbb{C}\right)
\]
since, the sole element tangent to $\mbox{Id}$ is the identity itself.
Therefore, $\mathcal{H}\left(f\right)$ can be seen as abelian subgroup
of $\mbox{Aff}\left(\mathbb{C}\right)$. Hence, the group has a fix
point and can be seen as a subgroup of the linear transformations
of $\mathbb{C}$. Now let us write the relation on the Schwartzian
seen at $\infty$
\[
f\left(1/\left(1/h\left(1/z\right)\right)\right)=\frac{1}{h^{'}\left(\frac{1}{z}\right)^{2}}f\left(\frac{1}{z}\right).
\]
Setting, $u\left(z\right)=\frac{1}{z^{4}}f\left(\frac{1}{z}\right)$
yields $u\left(az+b\right)=\frac{1}{a^{2}}u\left(z\right).$ Since,
$u=\frac{\alpha}{z^{4}}+\cdots$ we can consider the double primitive
function $U=\iint u$ with $U\left(\infty\right)=0$. This is a univalued
holomorphic function defined near $\infty$. Finally, the function
$U$ satisfies the following functionnal relation 
\[
U\left(az+b\right)=U\left(z\right).
\]
But in view of the dynamics of $\mbox{Lin}\left(\mathbb{C}\right)$,
it is clear that if $\mathcal{H}\left(f\right)$ is infinite then
$U=\mbox{cst}$ and thus $u=0$. 
\end{enumerate}
\end{proof}
In the course of the proof of the above result, we obtain the following
result
\begin{cor}
Let $\mathcal{M}$ be the quotient of $\mathbb{C}$ by the relation
$\alpha\sim\alpha^{'}$ if and only if $\mathcal{F}_{\alpha,\sigma}\sim\mathcal{F}_{\alpha^{'},\sigma}$
then there is only two possibilities 
\begin{enumerate}
\item $\mathcal{M}=\left\{ 0\right\} $ when $\sigma$ is an homography
- $\mathcal{F}_{\alpha,\sigma}$ is then analytically to $\frac{y^{2}+x^{3}}{xy}.$
\item $\mathcal{M}=\EQ{\mathbb{C}}H$ where $H$ is a finite subgroup of
$\mbox{Aff}\left(\mathbb{C}\right)$. 
\end{enumerate}

Genrerically, $H$ is reduced to $\left\{ \mbox{Id}\right\} $.

\end{cor}
As an obvious consequence, the functions obtained in proposition (\ref{prop:Let--be})
define two foliations analytically equivalent.

\subsection{Toward a geometric description of the moduli of Mattei. }

It remains to give a geometric interpretation of the parameter $\alpha$.
A promising approach is the following. Near the singular point of
the divisor, the leaf is conformally equivalent to a disc minus two
points which are the intersections between the leaf and the exceptionnal
divisor. If we consider in the leaf a path linking this two points,
we obtain after taking the image of this path by $E,$ an \emph{asymptotic
cycle} $\gamma$ as defined in \cite{Teyssier}which is not topologically
trivial. 

\noindent \begin{center}
\includegraphics[scale=0.5]{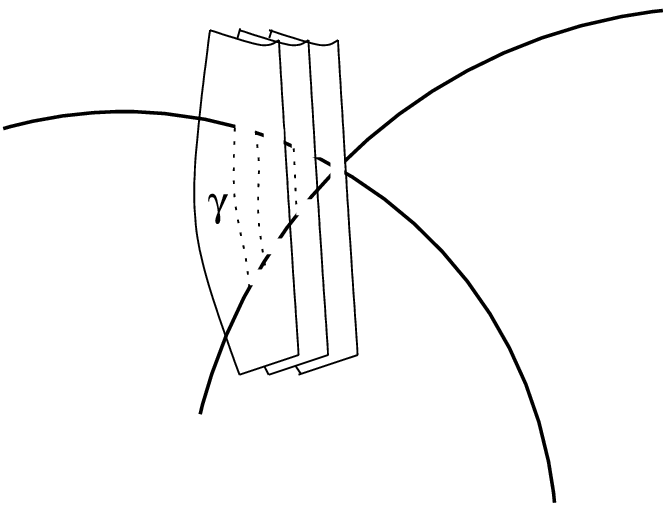}
\par\end{center}

Therefore, considering the family of these cycles parametrized by
a transversal parameter to the foliation yields a vanishing asymptotic
cycle. We claim that the moduli of Mattei should be associated to
the \emph{length} of this vanishing asymptotic cycle: more precisely,
it should be computed by the integral of some form along this vanishing
cycle. Actually, it easy to prove the following: let $\omega$ be
a $1-$form defining an absolutely dicritical foliation of cusp type
and let $\eta$ be any germ of $1$ form. Then $\eta$ is relatively
exact with respect to $\omega$, i.e., there exist two germs of holomorphic
functions $f$ and $g$ such that 
\[
\eta=df+g\omega
\]
if and only if the integral of $\eta$ along any asymptotic cycle
$\gamma$ vanish. Thus, we think that in a sense that has to be worked
out, the moduli of Mattei should be computed by the integral of some
generator of the relative cohomology group of $\omega$ along the
asymptotic vanishing cycle. 

\bibliographystyle{plain}
\bibliography{Bibliographie,biblio,dossier_scientifique}

\end{document}